\documentclass[12pt]{amsart}
\pdfoutput=1
\usepackage{amsmath,graphics}
\usepackage{amscd}
\usepackage{mathdots,amsfonts,amssymb,xypic}
\usepackage{mathrsfs}
\usepackage{graphicx}

\theoremstyle{plain}
 
\newtheorem{thm}{Theorem} 
\newtheorem{lemma}[thm]{Lemma}

\newtheorem{prop}[thm]{Proposition}

 \theoremstyle{remark}
\newtheorem{remark}[thm]{Remark}
\newtheorem*{claim}{Claim}

\newcommand{\al}{\alpha}
\newcommand{\be}{\beta}
\newcommand{\de}{\delta}
\newcommand{\ga}{\gamma}
\newcommand{\ep}{\varepsilon}

\newcommand{\si}{\sigma}
\newcommand{\De}{\Delta}

\newcommand{\Si}{\Sigma}

\newcommand{\SL}{{\rm SL}}
\newcommand{\GL}{{\rm GL}}

\newcommand{\ZZ}{{\mathbb Z}}
\newcommand{\FF}{{\mathbb F}}

\newcommand{\CC}{{\mathbb C}}
\newcommand{\NN}{{\mathbb N}}

\newcommand{\Aut}{\operatorname{Aut}}
\newcommand{\Hom}{\operatorname{Hom}}

\newcommand{\sm}{{\smallsetminus}}

\def\co{\colon\thinspace}

\def\modd{\, {\rm mod}\,}
\def\id{\operatorname{id}}

\begin{document}
\title[Metabelian SL$(n,\CC)$ representations of knot groups IV]{Metabelian SL$(n,\CC)$ representations of knot groups IV: Twisted Alexander polynomials}

\author{Hans U. Boden}
\address{Mathematics \& Statistics, McMaster University,
Hamilton, Ontario} \email{boden@mcmaster.ca}
\thanks{The first author was supported by a grant from the Natural Sciences and Engineering Research Council of Canada.}

\author{Stefan Friedl}
\address{Mathematisches Institut\\ Universit\"at zu K\"oln\\   Germany}
\email{sfriedl@gmail.com}


\date{September 18, 2012}
\begin{abstract}
In this paper we will study properties of twisted Alexander polynomials of knots corresponding to metabelian representations.
In particular we answer a question of Wada about the twisted Alexander polynomial associated to the tensor product of two representations, and we settle several conjectures of Hirasawa and Murasugi.
 \end{abstract}

\maketitle

%==========================================================
\section{Introduction}

Suppose $K\subset S^3$ is an  oriented knot.
We write $X_K=S^3\sm \nu K,$ where $\nu K$ denotes an open tubular neighborhood of $K$. Throughout the paper
we also write $\pi_K:=\pi_1(X_K)$ for the knot group.
Given  a representation $\al \co \pi_K\to \GL(n,\CC)$, Wada \cite{Wa94}, building on work of Lin \cite{Li01}, introduced an invariant $\De_K^\al(t)\in \CC(t)$.
This invariant is often referred to as the `twisted Alexander polynomial', `twisted torsion' or `Wada's invariant' of $(K,\al)$.
We refer to Section \ref{section:deftap} and \cite{Wa94,FV10} for more details.
If $\ep$ is the trivial rank one representation, then $\De_K^\ep(t)=\De_K(t)/(1-t)$, where $\De_K(t)$
denotes the classical Alexander polynomial of $K$.

%==========================================================
\subsection{Wada's question}

Let $K\subset S^3$ be an oriented knot  and let $\al$ and $\be$ be representations of $\pi_K$.
It follows easily from the definitions that
\[ \De_K^{\al\oplus \be}(t)=\De_K^\al(t) \cdot \De_K^\be(t),\]
i.e. the twisted Alexander polynomial $\De_K^{\al\oplus \be}(t)$ corresponding to the \emph{direct sum} $\al \oplus \be$ equals the product of the twisted Alexander polynomials $\De_K^{\al}(t)$ and $\De_K^{\be}(t)$ corresponding to the two representations $\al$ and $\be$.

Wada \cite{Wa11} asked whether something similar holds for  the twisted Alexander polynomial $\De_K^{\al\otimes \be}(t)$ associated with the \emph{tensor product} $\al \otimes \be$, i.e. is  $\De_K^{\al\otimes \be}(t)$  determined by the twisted Alexander polynomials $\De_K^{\al}(t)$ and $\De_K^{\be}(t)$ corresponding to the two representations $\al$ and $\be$? We answer this question in the negative. More precisely in
Section \ref{section:wada} we prove the following result:

\begin{thm}\label{thm:wadanegativeintro}
There exist two knots $K$ and $K'$, an isomorphism
of the metabelian quotient groups
\[ \Phi\co \pi_K/\pi_K^{(2)}\to \pi_{K'}/\pi_{K'}^{(2)},\]
and metabelian unitary representations $\al'$ and $\be'$ of $\pi_{K'}$ such that,
for the induced  metabelian representations $\al$ and $\be$ of $\pi_K$ given by
$\al =\al' \circ \Phi$ and $\be = \be' \circ \Phi,$ we have that
\[ \De_K^{\al}(t)=  \De_{K'}^{\al'}(t) \text{ and }  \De_K^{\be}(t)=   \De_{K'}^{\be'}(t), \text{ but }  \De_K^{\al \otimes \be}(t) \ne \De_{K'}^{\al'\otimes \be'}(t).\]
\end{thm}

%==========================================================
\subsection{The Hirasawa-Murasugi conjectures}

Hirasawa and Murasugi \cite{HM09,HM09b} studied in detail twisted Alexander polynomials corresponding to metabelian representations. They developed techniques for providing explicit computations of the twisted Alexander polynomials and stated several conjectures based on their results. We now recall their conjectures.

Given $n\in \NN$, let $\phi_n(t)$ be the $n$-th cyclotomic polynomial, i.e. set
\[\phi_n(t)=\prod_{\substack{k\in \{1,\dots,n\} \\ (k,n)=1}}(t-e^{2\pi ik/n})\in \ZZ[t^{\pm 1}].\]
Given a prime number $p$, let $\FF_p$ denote the finite field with $p$ elements
and set
\[ A_{p,n} = \FF_p[t^{\pm 1}]/(\phi_n(t)).\]
Note that $A_{p,n}$ is a finite group of order $p^k$ isomorphic to $(\ZZ/p)^k,$ where $k= \deg \phi_n(t)$.
There is an action of $\ZZ/n$ on $A_{p,n}$  defined by letting
 $i \in \ZZ/n$ act  by multiplication by $t^i$, and this action is well-defined because $\phi_n(t)|(t^n-1)$.
 We consider the semidirect product
\[ \ZZ/n\ltimes A_{p,n}.\]

Letting $A_{p,n}$  act on itself by addition, we obtain an action of
the semidirect product $\ZZ/n\ltimes A_{p,n}$ on the abelian group $A_{p,n}$. We then consider the resulting representation
\[ \ga \co  \ZZ/n\ltimes  A_{p,n} \to \Aut_\ZZ \left( \ZZ\left[A_{p,n}\right]\right).\]
The following conjecture was formulated by Murasugi and Hirasawa \cite[Conjecture~6.1]{HM09}:
\\

\noindent \textbf{Conjecture A (Hirasawa-Murasugi)}\emph{
Let $K$ be an oriented knot  together with an epimorphism
$\al \co \pi_1(X_K)\to \ZZ/n\ltimes A_{p,n},$ where $p$ is a prime.
Suppose that  $n$ and $p$ are coprime and that $\phi_n(t)$ is irreducible over $\FF_p[t^{\pm 1}]$.
 Then
\[ \De_{K}^{\ga \circ \al}(t)=\frac{\De_K(t)}{1-t} \cdot F(t),\]
where $F(t)$ is an integer polynomial in $t^n$.} \\

In Section \ref{section:conja} we will see that this conjecture also implies
 \cite[Conjecture~A]{HM09}.

\bigskip

 In \cite{HM09b}, Hirasawa and Murasugi also studied twisted Alexander polynomials corresponding to certain metacyclic representations, and we recall their further conjectures in this context.

To begin, we introduce the metacyclic groups denoted $G(m,p|k)$ in \cite{HM09b}.
Here $p$ is an odd prime, $m$ is a positive integer, and $k$ is an integer
which is a primitive $m$-th root of $1$ modulo $p$,
i.e. $k$ has the property that
\[ k^m\equiv 1\modd p \text{ but } k^\ell\not\equiv 1\modd p \text{  for } \ell=1,\dots,m-1.\]
Hirasawa and Murasugi  then define the group
\[ G(m,p|k):=\langle x,y \mid x^m=y^p=1\text{ and }xyx^{-1}=y^k\rangle.\]
Note that $G(m,p|k)$ is isomorphic to a  semidirect product of the form $\ZZ/m \ltimes \ZZ/p$, and there is a $G(m,p|k)$ action on $\ZZ/p$ given by
\[ y\cdot n=n-1 \text{ and } x\cdot n=kn \text{ for }n\in \ZZ/p.\]
This action defines an  embedding of  $G(m,p|k)$ into the symmetric group $S_p$ and hence in $\GL(p,\ZZ)$ via permutation matrices.
Letting $\varrho \co G(m,p|k)\to \GL(p,\ZZ)$ denote this representation,
we can now state Conjecture A of  \cite{HM09b}:
\\

\noindent \textbf{Conjecture A$\boldmath '$ (Hirasawa-Murasugi)}\emph{
Let $K$ be an oriented knot  together with an epimorphism
$\al \co \pi_1(X_K)\to G(m,p|k)$. Then
\[ \De_{K}^{\varrho \circ \al}(t)=\frac{\De_K(t)}{1-t} \cdot F(t),\]
where $F(t)$ is an integer polynomial in $t^m$.} \\

Taking $k=-1$ and $m=2$ in $G(m, p|k)$ just gives the dihedral group $D_p$ of order $2p$,
which is the group with presentation
\begin{equation} \label{eqndihedralgp}
D_p = \langle x,y \mid x^2=y^p=1\text{ and }xyx=y^{-1}\rangle.
\end{equation}
For the dihedral groups,
Hirasawa and Murasugi \cite[Conjecture~B]{HM09b}  proposed the following refinement of Conjecture A$'$:
\\

\noindent
\textbf{Conjecture B (Hirasawa-Murasugi)}\emph{
Let $K$ be an oriented knot  together with an epimorphism
$\al \co \pi_1(X_K)\to D_p$, where $p=2\ell+1$ is an odd prime.
\begin{enumerate}
\item There exists  an integer polynomial $f(t)$ such that
\[ \De_{K}^{\varrho \circ \al}(t)=\frac{\De_K(t)}{1-t} \cdot f(t)f(-t).\]
\item The following equality holds modulo $p$:
\[ \De_{K}^{\varrho \circ \al}(t)=\left(\frac{\De_K(t)}{1-t}\right)^{\ell+1}\cdot \left(\frac{\De_K(-t)}{1+t}\right)^\ell\, \modd p.\]
\end{enumerate}}

Note that the product $f(t)f(-t)$ is necessarily a polynomial in $t^2$. In this sense Conjecture B (1) is indeed a refinement of Conjecture A$'$ for epimorphisms onto dihedral groups.

Hirasawa and Murasugi provide extensive computational evidence for Conjectures A and  A$'$ in the papers \cite{HM09, HM09b}.
They also establish Conjecture B for certain two-bridge knots in \cite{HM09b}.
In  \cite{HS12}, Hoste and Shanahan give further  computations of the twisted Alexander polynomials
for many two-bridge and torus knots, and they provide additional supporting evidence for Conjecture B.

In Section \ref{section:hm} we settle the Hirasawa-Murasugi conjectures. More precisely, we prove the following result:

\begin{thm}
\begin{enumerate}
\item Conjecture A holds.
\item Conjecture A$'$ holds.
\item There are knots for which Conjecture B (1) fails.
\item Conjecture B (2) holds.
\end{enumerate}
\end{thm}

Although Conjecture B (1) fails for knots in general, it is entirely conceivable, in light of \cite{HM09b,HS12}, that
it holds for all 2-bridge knots.

After finishing this paper, the authors learned that Hirasawa and Murasugi have independently proved Conjecture A. Their proof will appear in a forthcoming revised version of \cite{HM09}.

%==========================================================
\subsection*{Organization of the paper}

In Section \ref{section:twialex} we first recall the definition and basic properties of twisted Alexander polynomials.
A key tool in the proofs of these results is the classification of metabelian representations of knot groups
(see \cite{BF08}), which we recall in Section \ref{section:metab}.
In Section \ref{section:wada}, we answer Wada's question, and in Section \ref{section:hm}, we address the Hirasawa-Murasugi conjectures.

\bigskip \noindent
{\em Acknowledgments. }
The authors would like to thank Taehee Kim for helpful conversations and also
Alexander Stoimenow for providing braid descriptions of all prime knots up to twelve crossings \cite{St12}.
The first author is grateful to the Max Planck Institute for Mathematics for its support.

%============================================================
\section{Twisted Alexander polynomials}\label{section:twialex}

%============================================================
\subsection{Definition and basic properties}\label{section:deftap}
We quickly recall the definition of twisted Alexander polynomials along the lines of Wada's paper \cite{Wa94}.
Let $K$ be an oriented knot, let $\phi\co \pi_K\to \ZZ$ be the epimorphism which sends a meridian to 1 and let $\al \co \pi_K\to \Aut(V)$ be a representation where $V$ is a finite dimensional free module over a unique factorization domain (UFD) $R$ with quotient field $Q$.
Note that $\al$ and $\phi$ give rise to a tensor representation
\[ \begin{array}{rcl} \al \otimes \phi \co \pi_K &\to & \Aut(V\otimes_{R}\ZZ[t^{\pm 1}]) \\
g&\mapsto & \left(v\otimes f(t)\mapsto \al(g)(v)\otimes t^{\phi(g)}f(t)\right).\end{array} \]
The map $\al\otimes \phi$ naturally extends to a map $\ZZ[\pi_K] \to \Aut(V\otimes_R\ZZ[t^{\pm 1}])$.
If $A$ is a matrix over $\ZZ[\pi_K]$ then we denote by $(\al\otimes \phi)(A)$ the matrix which is given by applying
$\al\otimes \phi$ to each entry of $A$.

Now let
\[ \pi_K=\langle g_1,\dots,g_{k+1} \mid r_1,\dots,r_{k}\rangle\]
be a presentation of $\pi_K$ of deficiency one. We denote by $F_{k+1}$ the free group with generators $g_1,\dots,g_{k+1}$.
Given $j\in \{1,\dots,k+1\}$ we denote by  $\frac{\partial }{\partial g_j}\co \ZZ[F_{k+1}]\to \ZZ[F_{k+1}]$ the Fox derivative with respect to $g_j$, i.e. the unique $\ZZ$-linear map such that
\begin{eqnarray*}
\frac{\partial g_i}{\partial g_j}&=&\de_{ij},\\
\frac{\partial uv}{\partial g_j}&=&\frac{\partial u}{\partial g_j}+u\frac{\partial v}{\partial g_j}
\end{eqnarray*}
for all $i,j \in \{1,\dots,k+1\}$ and $u,v\in F_{k+1}$.
We now denote by
\[M:=\left(\frac{\partial r_i}{\partial g_j}\right)\]
the $k\times (k+1)$-matrix over $\ZZ[\pi_K]$ which is given by all the Fox derivatives of the relators.
Furthermore, given $i\in \{1,\dots,k+1\}$ we denote by $M_i$ the  $k\times k$-matrix which is given by deleting the $i$-th column of $M$.

Note that there exists at least one $i$ such that $\phi(g_i)\ne 0$. It follows that  $\det((\al\otimes \phi)(1-g_i))\ne 0$.
Wada \cite{Wa94} then defined the twisted Alexander polynomial of $(K,\al)$ as follows:
\[ \De_K^\al(t) :=\det((\al\otimes \phi)(M_i))\cdot \det((\al\otimes \phi)(1-g_i))^{-1}\in Q(t).\]

\begin{remark}
The twisted Alexander polynomial of $(K,\al)$ was first introduced by Lin \cite{Li01} using a slightly different definition.
Several alternative definitions and interpretations were given by Kitano \cite{Ki96} and Kirk-Livingston \cite{KL99}.
In particular  $\De_K^\al(t)$ can be viewed as the Reidemeister torsion of a twisted complex,
which makes it possible to prove many structure theorems using the general techniques of Reidemeister torsion.
We refer to \cite{Ki96,FV10} for details.
\end{remark}

Wada \cite{Wa94} proved the following lemma.

\begin{lemma}\label{lem:wadadefined}
Let $K$ be an oriented knot and  let $\al \co \pi_K\to \Aut(V)$ be a representation where $V$ is a finite dimensional free module over a UFD $R$.
Then  $\De_{K}^\al(t)$ is well-defined up to multiplication by a factor of the form $\pm t^kr$, where $k\in \ZZ$ and
$r\in \det(\al(\pi_K))$.
\end{lemma}

In the following we write $\De_K^\al(t) \doteq f(t)$ if $\De_K^\al(t)$ and $f(t)$ agree up to the indeterminacy of
$\De_K^\al(t)$.

We now collect several well-known results about twisted Alexander polynomials.
Most of  the subsequent statements are immediate consequences of the definition and basic properties of determinants.
We  refer to \cite{Wa94,FV10} for details.

\begin{lemma}\label{lem:propstap}
Let $K$ be an oriented knot with classical Alexander polynomial $\De_K(t)$.
\begin{enumerate}
\item If $\ep \co \pi_K\to \GL(1,\ZZ)$ is the trivial representation, then
\[\De_K^\ep(t)\doteq \frac{\De_K(t)}{1-t}.\]
\item If $\tau \co \pi_K\to \GL(1,\CC)$ is a representation which is given by sending the oriented meridian to $z\in \CC\sm \{0\}$,
then
\[\De_K^\tau(t)\doteq \frac{\De_K(tz)}{1-tz}.\]
\item If $\al \co \pi_K\to \Aut_R(V)$ and $\be\co \pi_K\to \Aut_R(W)$ are isomorphic representations, i.e.
if there exists an isomorphism $\Phi \co V\to W$ such that $\al=\Phi^{-1}\circ \be\circ \Phi$, then
\[ \De_{K}^\al(t) \doteq \De_{K}^\be(t).\]
\item If  $\al \co \pi_K\to \GL(n,R)$, $R$ a UFD, is a representation of the form
\[ \al(g)=\begin{bmatrix} \be(g) & \ga(g) \\ 0& \de(g)\end{bmatrix}, \]
then
\[ \De_{K}^\al(t) \doteq \De_K^\be (t) \cdot \De_K^\de(t).\]
\item Let $\al \co \pi_K\to \Aut_{\ZZ}(W)$ be a representation
with $W$ a finitely generated free $\ZZ$-module and let $p$ be a prime. We denote by $\al_p \co \pi_K\to \Aut_{\FF_p}(W\otimes_{\ZZ}\FF_p)$  the  `mod $p$' reduction of $\al$. Then
\[ \De_K^{\al_p}(t)\equiv \De_K^{\al}(t) \modd p.\]
\end{enumerate}
\end{lemma}

%==========================================================
\subsection{Satellite knots}

Let $K\subset S^3$ be an  oriented knot and let
$C\subset S^3 $ be an oriented knot. Let $A\subset X_K$ be a simple closed
curve, unknotted in $S^3$. Then $S^3\sm \nu A$ is a solid torus. Let
$\phi\co \partial(\overline{\nu A})\to
\partial(\overline{\nu C})$ be a diffeomorphism which sends a meridian
of $A$ to a longitude of $C$, and a longitude of $A$ to a meridian of
$C$. The space
\[
\left({S^3\sm \nu A}\right) \cup_{\phi} \left({S^3\sm \nu C}\right)
\]
is diffeomorphic to $S^3$. The image of $K$ is denoted by
$S=S(K,C,A)$. We say $S$ is the \emph{satellite knot with companion
  $C$, orbit $K$ and axis $A$}. Put differently, $S$ is the result of
replacing a tubular neighborhood of $C$ by an oriented knot  in a solid torus,
namely by $K\subset {S^3\sm \nu A}$. Note that $S$ inherits an
orientation from $K$.

 The abelianization map $\pi_1(S^3\sm \nu C)\to \ZZ$ gives rise to
a degree one map from $S^3\sm \nu C$ to $\overline{\nu A}$ which is a
diffeomorphism on the boundary.  In particular we get an induced map
\[
X_S = \left({S^3\sm \nu A \sm \nu K}\right) \cup_{\phi} \left({S^3\sm
    \nu C}\right) \to \left({S^3\sm \nu A \sm \nu K}\right) \cup
\overline{\nu A}= X_K
\]
which we denote by $f$. Note that $f$ is a diffeomorphism on the
boundary and that $f$ induces an isomorphism of homology groups.
Also note that the curve $A$ determines an
element $[A]\in \pi_K$ which is well--defined up to conjugation.

Given a group $\pi$  we denote by  $\pi^{(n)}$  the $n$--th term of the
derived series of $\pi$. These subgroups are defined inductively by setting
$\pi^{(0)}=\pi$ and $\pi^{(i+1)}=[\pi^{(i)},\pi^{(i)}]$.
The following lemma is well-known and follows from a standard Seifert-van Kampen argument.

\begin{lemma}\label{lem:samealexandermodule}
 Let $K,C,A, S=S(K,C,A)$ and $f\co X_S\to X_K$ as
  above. If $[A]$ lies in $\pi_K^{(n)}$, then $f$ induces an isomorphism
  \[ \pi_S/\pi_S^{(n+1)} \xrightarrow{\cong } \pi_K/\pi_K^{(n+1)}.\]
\end{lemma}

The next lemma is proved by using the reinterpretation of Wada's twisted Alexander polynomial as twisted Reidemeister torsion
(see \cite{Ki96,FV10}) and standard `Mayer-Vietoris-style' arguments. We refer to \cite[Lemma~7.1]{CF10} for a proof of a more general statement.

\begin{lemma} \label{lemmahofsat} Let $K,C,A, S=S(K,C,A)$ and $f$ as
  above. We suppose that $A$ is null-homologous in $X_K$.
    Let $\al\co \pi_K\to
  \GL(k,\CC)$ be a representation. We denote the representation
  $\pi_S\xrightarrow{f}\pi_K\xrightarrow{\al} \GL(k,\CC)$ by $\al$ as
  well.  Denote by $z_1,\dots,z_k$ the eigenvalues of
  $\al(A)$ and let $\De_C(t)\in \ZZ[t^{\pm 1}]$ be a fixed representative of the
  Alexander polynomial of $C$.  Then the following holds:
  \[ \De_S^{\al}(t)\doteq \De^{\al}_K(t)\cdot
    \prod_{i=1}^k \De_C(z_i) \in \CC(t).\]
\end{lemma}

%==========================================================
\section{Metabelian representations of knot groups}\label{section:metab}

A  representation $\al$ of a group $\pi$
is called \emph{metabelian}
 if $\al$ factors through $\pi/\pi^{(2)}$.
In this section we recall some results from \cite{Fr04,BF08, BF11, BF12}
regarding  metabelian  representations of knot groups.

%===========================================================
\subsection{Classification of irreducible metabelian representations of knot groups}

Let $K\subset S^3$ be a knot. We write $\pi=\pi_K$ and we denote by $H=H_1(X_K;\ZZ[t^{\pm 1}])$ its Alexander module.
It is well-known that $\pi/\pi^{(2)}$ is isomorphic to $\ZZ\ltimes H$, where $n\in \ZZ$ acts on $H$ by multiplication by $t^n$.
(See e.g. \cite{Fr04} for a proof.)
Let
$\chi \co H\to \CC^*$ be a character which factors through $H/(t^n-1)$ and $z\in U(1)$. Then it follows from
\cite[Section~3]{BF08} that, for $(j, h) \in \ZZ\ltimes H,$ setting
$$  \al_{(n,\chi,z)} (j,h) =
 \begin{bmatrix}
 0& &\dots &z \\
 z&0&\dots &0 \\
\vdots &\ddots &\ddots&\vdots \\
     0&\dots &z &0 \end{bmatrix}^j
     \begin{bmatrix} \chi(h) &0&\dots &0 \\
 0&\chi(th) &\dots &0 \\
\vdots &&\ddots &\vdots \\ 0&0&\dots &\chi(t^{n-1}h) \end{bmatrix}
$$
defines a  $\GL(n,\CC)$ representation.
Note that $\al_{(n,\chi,z)}$ factors through $\ZZ\ltimes H/(t^n-1)$.
We denote by $\al_{(n,\chi,z)}$ also the induced representation of $\pi$ obtained by precomposing
$\al_{(n,\chi,z)}$ with the epimorphism $\pi\to \pi/\pi^{(2)}\cong \ZZ\ltimes H$.

Now suppose  $z\in \CC$ satisfies $z^n=(-1)^{n+1}$ and set $\al_{(n,\chi)}=\al_{(n,\chi,z)}$.
Note that $\al_{(n,\chi)}$ defines a representation
$\al_{(n,\chi)}\co \pi_K\to \SL(n,\CC)$ and that  the isomorphism type of this representation is independent of the choice of $z$.

In \cite{BF08}, extending earlier work in \cite{Fr04}, we classified irreducible metabelian $\SL(n,\CC)$ representations of knot groups.
In particular we proved that given  any irreducible representation  $\al \co\ZZ\ltimes H\to \SL(n,\CC)$
 there exists a character $\chi \co H\to \CC^*$ which factors through $H/(t^n-1)$ such that
$\al$ is isomorphic to $\al_{(n,\chi)}$.

We conclude this section with the following proposition,
which is an immediate consequence of \cite{Ki96},  \cite[Proposition~1]{FV10} and  \cite[Proposition~5]{BF11}.

\begin{prop} \label{prop:taptm}
 Let $K$ be a knot and $\al\co \pi \to \SL(n,\CC)$ an irreducible  metabelian representation
 with $n>1$. Then $\De_K^\al(t)$ lies in $\CC[t^{\pm 1}]$ and it is a polynomial in $t^n$.
\end{prop}

Note that this proposition was also proved by Herald, Kirk, and  Livingston \cite[p.~935]{HKL10}.

%==========================================================
\subsection{Tensor products}

In the next section we will consider tensor products of  metabelian representations.
We will make use of the following proposition.

  \begin{prop} \label{prop:tensor}
Let $K\subset S^3$ be a knot and  let $H=H_1(X_K;\ZZ[t^{\pm 1}])$ be its Alexander module.
Let
$\chi_i \co H\to \CC^*, i=1,2$ be  characters which factor through $H/(t^{k_i}-1), i=1,2$.
If $k_1,k_2$ are coprime, then
\[ \al_{(k_1,\chi_1)} \otimes \al_{(k_2,\chi_2)}
\cong  \al_{(k_1k_2,\chi_1 \chi_2)}. \]
\end{prop}

For unitary representations, this proposition was stated and proved as Proposition 4.6 in \cite{Fr04}.
The proof carries over to the general case, and we quickly outline the argument for the reader's convenience.

\begin{proof}
 Denote by $e_{11},\dots,e_{k_11}$ and $e_{12},\dots,e_{k_22}$
the canonical bases of
$\CC^{k_1}$ and $\CC^{k_2}$. Set
$f_i:=e_{i \modd k_1,1} \otimes e_{i \modd k_2,2}$
for $i=0,\dots,k_1k_2-1$.
Since $k_1$ and $k_2$ are coprime it follows that the $f_i$'s are distinct. In particular $\{f_i\}_{
i=0,\dots,k_1k_2-1}$ form a  basis for $\CC^{k_1}\otimes \CC^{k_2}$.
One can easily see that the representation $\al_{(k_1,\chi_1)} \otimes \al_{(k_2,\chi_2)}$ with respect to this basis
is just $\al_{(k_1k_2,\chi_1 \chi_2)}$.
\end{proof}

%==========================================================
\section{Wada's question}\label{section:wada}

Let $K\subset S^3$ be an oriented knot  and let $\al$ and $\be$ be representations of $\pi_K$.
Wada \cite{Wa11} asked whether the twisted  Alexander polynomial $\De_K^{\al\otimes \be}(t)$ is determined by $\De_K^\al(t)$ and $\De_K^\be(t)$.

The following theorem gives a negative answer to
 any reasonable interpretation of Wada's question.

\begin{thm}\label{thm:wadanegative}
There exist two knots $S$ and $S'$, an isomorphism of the metabelian quotient groups
\[ \Phi\co \pi_S/\pi_S^{(2)}\to \pi_{S'}/\pi_{S'}^{(2)},\]
and metabelian unitary representations $\al'$ and $\be'$ of $\pi_{S'}$ such that,
for the induced  metabelian representations $\al$ and $\be$ of $\pi_S$ given by
$\al =\al' \circ \Phi$ and $\be = \be' \circ \Phi,$ we have that
\[ \De_S^{\al}(t) \doteq  \De_{S'}^{\al'}(t) \text{ and } \De_S^{\be}(t) \doteq   \De_{S'}^{\be'}(t),\text{ but } \De_S^{\al \otimes \be}(t) \not\doteq \De_{S'}^{\al'\otimes \be'}(t).\]
\end{thm}

We begin by outlining the strategy for proving this theorem.
Let $K$ be an oriented knot, let  $A\subset S^3\sm \nu K$ a curve, unknotted in $S^3$, which is null-homologous in $S^3\sm \nu K$.
Let $C$ and $C'$ be two knots.
We write $S=S(K,C,A)$ and $S'=S(K,C',A)$.
By Lemma \ref{lem:samealexandermodule} we have isomorphisms
\[ \xymatrix{  \pi_S/\pi_S^{(2)} \ar[rd]^{f_*}_\cong &&\pi_{S'}/\pi_{S'}^{(2)} \ar[dl]_{f'_*}^\cong\\
&\pi_K/\pi_K^{(2)},}\]
and we set $\Phi = (f'_*)^{-1} \circ f_*$.

We write $H=H_1(X_K;\ZZ[t^{\pm 1}])$. Let   $n_1,n_2$ be coprime numbers and let $\chi_i\co H/(t^{n_i}-1) \to S^1, i=1,2$ be two characters.
 We write $\al_i=\al(n_i,\chi_i)$. By a slight abuse of notation we denote the induced representations $\al_i\circ f_*$ and
 $\al_i\circ f_*'$ by $\al_i$ as well.
Recall that $\al_i, i=1,2$ are special linear representations, it thus follows from
Lemma \ref{lem:wadadefined}  that the corresponding twisted Alexander polynomials are well-defined up to multiplication
by a factor of the form $\pm t^i,i\in \ZZ$.
 We furthermore write
 \begin{eqnarray*} Z_1&:=&\{ \chi_1(A),\chi_1(tA),\dots,\chi_1(t^{n_1-1}A)\}, \\
 Z_2&:=&\{ \chi_2(A),\chi_2(tA),\dots,\chi_2(t^{n_2-1}A)\}, \\
 Z&:=&\{z_1\cdot z_2 \mid z_1\in Z_1\text{ and } z_2\in Z_2\} \\
 &=&\{ (\chi_1\chi_2)(A),\dots,(\chi_1\chi_2)(t^{n_1n_2-1}A)\},
 \end{eqnarray*}
where these sets consist of complex numbers counted according to multiplicity.
It follows from Proposition \ref{prop:tensor} and Lemma  \ref{lemmahofsat}
that
\begin{eqnarray*}
 \De_S^{\al_1}(t)&\doteq& \De^{\al_1}_K(t) \cdot \prod_{z_1\in Z_1} \De_C(z_1), \\
    \De_S^{\al_2}(t)&\doteq& \De^{\al_2}_K(t) \cdot \prod_{z_2\in Z_2} \De_C(z_2),\\
 \De_S^{\al_1\otimes \al_2}(t)&\doteq& \De^{\al_1\otimes \al_2}_K(t) \cdot \prod_{z\in Z} \De_C(z),\end{eqnarray*}
and similarly for $S'$.

Thus, it suffices to find two knots $C$ and $C'$ and integers $n_1,n_2$  such that,
where products are taken with multiplicities,
\begin{eqnarray*}
\prod_{z_1\in Z_1} \De_C(z_1) &=& \pm \prod_{z_1\in Z_1} \De_{C'}(z_1), \\
\prod_{z_2\in Z_2} \De_C(z_2) &=& \pm \prod_{z_2\in Z_2} \De_{C'}(z_2), \\
\prod_{z\in Z} \De_C(z) &\neq& \pm \prod_{z\in Z} \De_{C'}(z).
\end{eqnarray*}

\begin{lemma} \label{lem:trefoil}
Let $K$ be  the trefoil knot and let $H=H_1(X_K;\ZZ[t^{\pm 1}])$ be its Alexander module.
There exists a curve   $A\subset S^3\sm \nu K$, which is unknotted in $S^3$
 and which is null-homologous in $S^3 \sm \nu K$, and characters $\chi_1 \co H/(t^2-1)\to \ZZ/3\to S^1$ and
 $\chi_2 \co H/(t^3-1)\to \ZZ/2\to S^1$ such that
 \[ Z_1=\{e^{2\pi i/3},e^{-2\pi i/3}\} \text{ and } Z_2=\{-1,-1, 1\}.\]
\end{lemma}

\begin{proof}

We start out by recalling several well-known facts from knot theory.
Let $K$ be a knot and let $H=H_1(X_K;\ZZ[t^{\pm 1}])$ its Alexander module.
Given $k\in \NN$ we denote by  $L_k$ the $k$-fold branched cover of $K$. Note that the cyclic group
$C_k = \langle t\,|\,t^k=1\rangle$ naturally acts on $H_1(L_k;\ZZ)$.
Also note that  there exists a canonical isomorphism $$H/(t^k-1)\xrightarrow{\cong} H_1(L_k;\ZZ)$$ which is equivariant
with respect to the $C_k$ action.

Now suppose that $K$ is a fibered knot with fiber $\Si$ of genus $g$ and monodromy map $\varphi \co \Si \to \Si$.
We pick a basis for $H_1(\Si;\ZZ)$, i.e. we pick an identification $H_1(\Si;\ZZ)=\ZZ^{2g}$, and we denote by $V$ the corresponding Seifert matrix.

The induced map $\varphi_* \co H_1(\Si;\ZZ)\to H_1(\Si;\ZZ)$ on homology is
represented by the matrix $M:=V^{-1}V^t$, and for any $k$, we have the commutative diagram
\[ \xymatrix{ \ZZ^{2g}\ar[r]^{\cong} \ar[d] & H_1(\Si;\ZZ)\ar[d] \\ \ZZ^{2g}/(\id-M^k)\ZZ^{2g}\ar[r]^-{\cong} &H_1(L_k;\ZZ),} \]
which gives a canonical identification of $H_1(L_k;\ZZ)$ with $ \ZZ^{2g}/(\id-M^k)\ZZ^{2g}$. We also obtain the following commutative diagram
\[ \xymatrix{ \ZZ^{2g}/(\id-M^k) \ZZ^{2g}\ar[r]^-{\cong} \ar[d]^{\cdot M}&H_1(L_k;\ZZ)\ar[d]^{\cdot t}\\
\ZZ^{2g}/(\id-M^k) \ZZ^{2g} \ar[r]^-{\cong} &H_1(L_k;\ZZ).}\]

We now specialize to the case that $K$ is the trefoil knot, and
we let $\Si$ be the fiber of the genus one fibration $S^3\sm \nu K\to S^1$.
With an appropriate identification of  $H_1(\Si;\ZZ)$ with $\ZZ^2$, the Seifert matrix for $K$ is given by
\[ V=\begin{bmatrix}-1&0 \\ -1&-1\end{bmatrix}.\]

We first consider $H_1(L_2;\ZZ)$. By the above we can identify $H_1(L_2;\ZZ)$ with $\ZZ^2/(\id-M^2)\ZZ^2$,
where
\[ M:=V^{-1}V^t=\begin{bmatrix} 1& 1 \\ -1&0\end{bmatrix}.\]
It is straightforward to see that $H_1(L_2;\ZZ)\cong \ZZ/3$ and that $e_1=(1,0)$ is a non-trivial element.
We now pick $\chi_1 \co H_1(L_2;\ZZ)\to \ZZ/3\to S^1$ such that $\chi_1(e_1)=e^{2\pi i/3}$.
It is straightforward to verify that $\chi_1(Me_1)=e^{-2\pi i/3}$.
 
We now turn to $H_1(L_3;\ZZ)$.
Since \[M^3=\begin{bmatrix} -1&0\\ 0&-1\end{bmatrix},\]
we see immediately that $H_1(L_3;\ZZ) = \ZZ/2 \oplus \ZZ/2$. Further,
we have the following commutative diagram
\[ \xymatrix{
 \ZZ^2/(\id-M^3)\ZZ^2\ar[d]^{\cdot M}
 \ar[r]^-{\cong}
&H_1(L_3;\ZZ)\ar[d]^{\cdot t}
\\
\ZZ^2/(\id-M^3)\ZZ^2 \ar[r]^-{\cong}
& H_1(L_3;\ZZ)  }\]

We now denote by $\chi_2$ the character which is given by
\[ H_1(L_3;\ZZ)  \cong \ZZ/2\oplus \ZZ/2\to \ZZ/2\to S^1,\]
where the second map is  the projection on the first factor and the third map is given by sending $1 \in \ZZ/2$ to $-1 \in S^1$.
Thus $\chi_2(e_1)=-1$, and it follows that $\chi_2(Me_1)=-1$ and $\chi_2(M^2e_1)=1$.

We pick a simple closed curve $C$ on $\Si \subset S^3\sm \nu K$ corresponding to  $(1,0)\in H_1(\Si;\ZZ)=\ZZ^2$.
Note that $C$ is null--homologous in $S^3\sm \nu K$. Also note that after crossing changes we can find a simple closed curve
$A\subset S^3\sm \nu K$ which is unknotted in $S^3$ but which is homotopic in $S^3\sm \nu K$ to $C$.
It follows from the above that
\[ \begin{array}{rclrcll} \chi_1([A])&=&e^{2\pi i/3},& \chi_1(t[A])&=&e^{-2\pi i/3}& \text{ and }\\
\chi_2([A])&=&-1,& \chi_2(t[A])&=& -1, &\chi_2(t^2[A])=1.\end{array}\]
This completes the proof of the lemma. \end{proof}

We now take $n_1=2$ and $n_2=3$ and consider the knots  $C=9_{30}$ and $C' = 11a359$, which have Alexander polynomials given by
\begin{eqnarray*}
\De_{C}(t) &=& 1 - 5t + 12t^2 - 17t^3 + 12t^4 - 5t^5 + t^6, \\
\De_{C'}(t) &=& 6 - 13t + 15t^2 - 13t^3 + 6t^4.
\end{eqnarray*}
We pause to explain a key property about the branched covers of these knots underlying this calculation.
Letting $L_k$ and $L'_k$ denote the $k$-fold branched cover along $C$ and $C'$, respectively,
using \cite{CL12}, one can easily determine that
\begin{eqnarray*}
H_1(L_2;\ZZ) =  &\ZZ/53 &=H_1(L'_2;\ZZ), \\
H_1(L_3;\ZZ) =  &\ZZ/22 \oplus \ZZ/22& =H_1(L'_3; \ZZ),
\end{eqnarray*}
but that
\begin{eqnarray*}
H_1(L_6;\ZZ) &=& \ZZ/2 \oplus \ZZ/2 \oplus \ZZ/22 \oplus \ZZ/1166, \\
H_1(L'_6;\ZZ) &=& \ZZ/88 \oplus \ZZ/4664
\end{eqnarray*}
have different orders.

Using the sets $Z_1 = \{ e^{2 \pi i/3}, e^{-2 \pi i/3} \}$ and $Z_2=\{ -1, -1, 1\}$ from Lemma \ref{lem:trefoil}
and setting $Z  = \{z_1áz_2 \mid z_1 \in Z_ 1 \text{ and } z_2\in Z_2\}$,
after taking into account multiplicities, we compute that
\begin{eqnarray*}
  \prod_{z_1 \in Z_1} \De_{C}(z_1)\;=&484&=  \;\prod_{z_1 \in Z_1} \De_{C'}(z_1), \\
  \prod_{z_2 \in Z_2} \De_{C}(z_2)\;=&\pm 2809&=  \; \prod_{z_2 \in Z_2} \De_{C'}(z_2),
  \end{eqnarray*}
but that
\begin{eqnarray*}   \prod_{z \in Z} \De_{C}(z) &=&937024 \\
   \prod_{z \in Z} \De_{C'}(z) &=& 3748096.
   \end{eqnarray*}

We have now showed that  $S=S(K,C,A)$ and $S'=S(K,C',A)$ where  $K$ is the trefoil, $A$ is the curve given  Lemma \ref{lem:trefoil}
and $C=9_{30}$ and $C' = 11a359$, together with the representations $\al_1=\al(2,\chi_1)$ and $\al_2=\al(3,\chi_2)$
provide a negative answer to Wada's question. This now   completes our discussion of Theorem \ref{thm:wadanegative}.

%==========================================================
\section{The Hirasawa-Murasugi conjectures}\label{section:hm}

%==========================================================
\subsection{Conjecture A}\label{section:conja}

The following gives a proof of Conjecture A stated in the introduction. Recall that $\phi_n(t)$ is the
$n$-th cyclotomic polynomial and we set $A_{p,n} = \FF_p[t^{\pm 1}]/(\phi_n(t))$.

\begin{prop}\label{prop:conja}
Let $K$ be an oriented knot  together with an epimorphism
$\al \co \pi_1(X_K)\to \ZZ/n\ltimes A_{p,n}$, where $p$ is a prime.
Suppose that  $n$ and $p$ are coprime and  that $\phi_n(t)$ is irreducible over $\FF_p[t^{\pm 1}]$.
 Then
\[ \De_{K}^{\ga \circ \al}(t)=\frac{\De_K(t)}{1-t} \cdot F(t),\]
where $F(t)$ is an integer polynomial in $t^n$.
\end{prop}

\begin{remark}
The original statement of \cite[Conjecture~A]{HM09}  is slightly different, in that it is  a statement about two bridge knots and twisted Alexander polynomials of metabelian representations that factor through the alternating group $A_4.$ More specifically,
in  \cite[Conjecture~A]{HM09} Hirasawa and Murasugi  consider twisted Alexander polynomials corresponding
to representations of the form $\de \circ \al\co \pi_1(X_K)\to  A_4\to \GL(4,\ZZ)$,
where $\al$ is assumed to be an epimorphism and $\de$ is the canonical representation given by permutation matrices.
It is straightforward to see  that
\[ A_4\cong \ZZ_3\ltimes \FF_2[t]/(t^2+t+1),\]
and that under this isomorphism $\de\cong \ga$.
It now follows that Proposition \ref{prop:conja} implies  \cite[Conjecture~A]{HM09} as a special case.
\end{remark}

\begin{proof}
We start out with the following claim:

\begin{claim}
The action of $\ZZ/n$ is free on nonzero elements of $A_{p,n}$.
\end{claim}

Suppose $a(t)\in A_{p,n}$ is a nonzero element and let $k\in \NN$ be the smallest $k$ with $t^ka(t)=a(t)$.
Note that $k$ necessarily divides $n$. We want to show that $k=n$.

First note that $a(t)(t^k-1)=0\in A_{p,n} = \FF_p[t^{\pm 1}]/(\phi_n(t))$ implies that there exists $b(t)\in \FF_p[t^{\pm 1}]$ such that
\[ a(t)(t^k-1)=b(t)\phi_n(t)\in \FF_p[t^{\pm 1}].\]
Note that over $\ZZ[t^{\pm 1}]$ we have $t^k-1= \prod_{d|k} \phi_d(t)$, in particular we see that
\[a(t)\prod_{d|k} \phi_d(t)=b(t)\phi_n(t)\in \FF_p[t^{\pm 1}].\]
By assumption $\phi_n(t)$ is irreducible over $\FF_p[t^{\pm 1}]$.
Since $\FF_p[t^{\pm 1}]$ is a UFD it follows that $\phi_n(t)$ is a prime element in $\FF_p[t^{\pm 1}]$.
Since $a(t)\ne 0\in \FF_p[t^{\pm 1}]/(\phi_n(t))$ it follows that $\phi_n(t)$ divides $\phi_d(t)$ in $\FF_p[t^{\pm 1}]$ for some $d|k$.
Comparing degrees and using that $k|n$  we see that the Euler function applied to $n,d$ and $k$ satisfies
\[ \varphi(n) \leq \varphi(d)\leq \varphi(k)\leq \varphi(n).\]
This is only possible if $n=d=k$.
This concludes the proof of the claim.

We  denote by $X_{p,n}:=\Hom(A_{p,n},S^1)$ the set of all characters on the finite abelian group $A_{p,n}$.
Let $\chi\in X_{p,n}$. We denote by $\CC_\chi$ the one dimensional complex vector space
viewed with the $A_{p,n}$-action given by $\chi$. It is well-known that
\begin{equation} \label{equ:c1}
\CC[A_{p,n}]\cong \bigoplus_{\chi \in X_{p,n}}\CC_\chi
\end{equation}
as $\CC[A_{p,n}]$-modules.

Given $k\in \ZZ$ we denote by $t^k\chi$ the character given by $t^k\chi(v):=\chi(t^kv)$.
Note that this defines an action by $\ZZ/n$ on $X_{p,n}$.
We now denote by $\chi_0$ the trivial character on $A_{p,n}$ and we write $X'_{p,n}:=X_{p,n} \sm \chi_0$.
It follows from the above claim that $\ZZ/n$ acts freely on $X'_{p,n}$.
We now pick coset representatives $\chi_1,\dots,\chi_\ell$ for the free $\ZZ/n$-action on $X'_{p,n}$.

We consider
\[ \begin{array}{rcl} V_0&=&\CC_{\chi_0},\\
V_1&=&\CC_{\chi_1}\oplus \CC_{t\chi_1}\oplus \dots \oplus \CC_{t^{n-1}\chi_1},\\
\vdots &&\vdots \\
V_\ell&=&\CC_{\chi_l}\oplus \CC_{t\chi_\ell}\oplus \dots \oplus \CC_{t^{n-1}\chi_\ell}.\end{array} \]
These are naturally modules over $\CC[A_{p,n}]$.
We furthermore equip $V_0$ with the trivial $\ZZ/n$ action and we let $\ZZ/n$ act on each $V_i$, $i=1,\dots,\ell$ by cyclically permuting the summands.
It is now straightforward to see that this turns $V_0,V_1,\dots,V_\ell$ into
$\CC[\ZZ/n \ltimes A_{p,n}]$-modules.
Note that $\ga$ also turns  $\CC[A_{p,n}]$ into a $\CC[\ZZ/n\ltimes A_{p,n}]$-module.
 It now  follows from \eqref{equ:c1}
that
\[ \CC[A_{p,n}]\cong V_0\oplus V_1\oplus \dots \oplus V_\ell\]
as $\CC[\ZZ/n \ltimes A_{p,n}]$-modules. We denote by $\ga_0,\ga_1,\dots,\ga_\ell$ the representations
corresponding to $V_0,V_1,\dots,V_\ell$. We can thus restate the above isomorphism as
$\ga\cong \ga_0\oplus \ga_1\oplus \dots \oplus \ga_\ell$.

Note that $\ga_0$ is the trivial one-dimensional representation.
Furthermore it follows from \cite[Lemma~2.2]{BF08} that $\ga_1,\dots,\ga_\ell$ are irreducible representations.
It thus follows from Lemma \ref{lem:propstap} combined with Proposition \ref{prop:taptm}  that
\begin{equation} \label{equ:tauk} \De_{K}^{\ga \circ \al}(t)=\prod\limits_{i=0}^\ell\De_{K}^{\ga_i\circ \al}(t)=\frac{\De_K(t)}{1-t} \cdot F(t)\end{equation}
where $F(t)$ is a complex  polynomial in $t^m$. It remains to show the following claim:

\begin{claim}
$F(t)$ is  an integral polynomial.
\end{claim}

First note that by definition $\De_K^{\varrho\circ \al}(t)$ can be written as a quotient of two polynomials $x(t)$ and $y(t)$ with integral coefficients
and such that $y(t)$  is a monic polynomial.
Furthermore it follows from \eqref{equ:tauk} that $\De_K^{\varrho\circ \al}(t)\cdot (1-t)=\frac{x(t)}{y(t)}\cdot (1-t)$ is a rational polynomial. Since $y(t)$ is monic this implies that
$\De_K^{\varrho\circ \al}(t)\cdot (1-t)$ is in fact an integral polynomial.

We now see that $F(t)\cdot \De_K(t)$ is an integral polynomial. It follows from the Gauss Lemma and from the fact that $\De_K(1)=\pm 1$ that $F(t)$ is also an integral polynomial.
This concludes the proof of the claim and also the proposition.
\end{proof}

\begin{remark}
Similar arguments can be used to establish Conjecture A$^\prime$ from the introduction.
We leave the details to the reader as a straightforward exercise.
\end{remark}

%==========================================================
\subsection{Conjecture B (1)}

We start out with the following observation which shows that a weaker version of Conjecture B (1) holds.

\begin{lemma}\label{lem:conjboverc}
Let $K$ be an oriented knot.
Let $p$ be an odd prime and let $\al\co \pi_K\to D_p$ be an epimorphism. Then
there exists a \emph{complex} polynomial $f(t)$ such that
\[ \De_{K}^{\varrho \circ \al}(t)=\frac{\De_K(t)}{1-t} \cdot f(t)f(-t).\]
\end{lemma}

\begin{proof}
It follows immediately from Proposition \ref{prop:conja} that  there
exists an integer polynomial $g(t)$ with
\[ \De_{K}^\al(t)=\frac{\De_K(t)}{1-t} \cdot g(t^2).\]
We can factor the polynomial $g(s)$ as follows:
\[ g(s)=C\cdot \prod_{i=1}^r(\al_i-s),\]
for some $C,\al_1,\dots,\al_r\in \CC$. We now pick square roots for $C,\al_1,\dots,\al_r$.
It then follows that
\[ g(t^2)=C\cdot \prod_{i=1}^r(\al_i-t^2)=\sqrt{C}\prod_{i=1}^r(\sqrt{\al_i}+t)\cdot \sqrt{C}\prod_{i=1}^r(\sqrt{\al_i}-t).\]
The complex polynomial $f(t)=\sqrt{C}\prod_{i=1}^r(\sqrt{\al_i}+t)$ thus has the required property.
\end{proof}

We propose a counterexample to Conjecture B (1).
We consider the knot $K=10_{164}$.
The untwisted Alexander polynomial equals
\[ \De_K(t)=3-11t+17t^2-11t^3+3t^4.\]
According to \cite{St12}, the knot $K$ is the closure of the braid on four strands
given in terms of the standard generators as
$$\be = \si_1 \si_2^{-1} \si_3^2 \si_2^{-1}\si_1 \si_2^{-1} \si_3^{-1}\si_2^{-1}\si_1 \si_2^{-1}.$$

This gives rise to a presentation of $\pi_K  = \pi_1(S^3\sm \nu K)$
with generators $a, b, c, d, e, f, g, h, i, j, k$
and relators
\[ \begin{array}{lllllllll}
&b^{-1}a^{-1}ea,&a^{-1}cfc^{-1},&d^{-1}f^{-1}gf,&f^{-1}g^{-1}hg,&c^{-1}hih^{-1},\\
&h^{-1}e^{-1}je,&e^{-1}iki^{-1},&k^{-1}gdg^{-1},&i^{-1}gbg^{-1},&g^{-1}j^{-1}aj.
\end{array} \]

\begin{figure}[ht]
\begin{center}
\leavevmode\hbox{}
\includegraphics[scale=0.90]{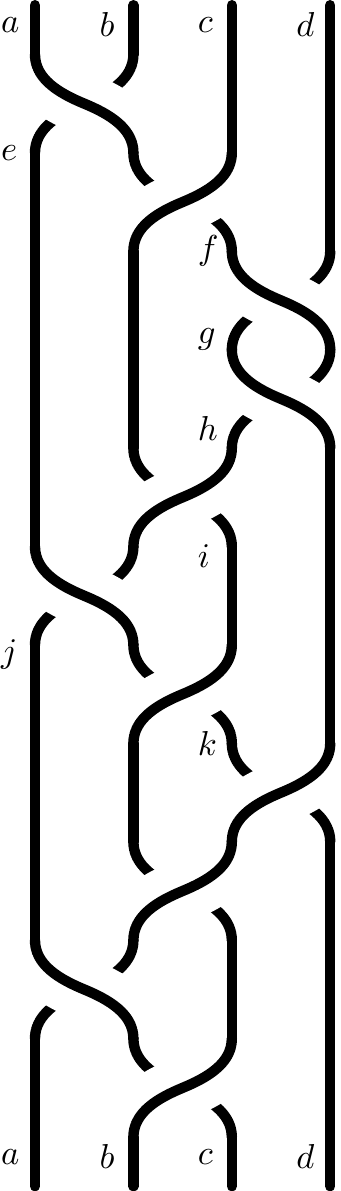}
\caption{A braid whose closure is the knot $K=10_{164}$ with the Wirtinger generators for the knot group $\pi_K$.} \label{braid}
\end{center}
\end{figure}

It can easily be checked that one gets a representation
\[ \al\co \pi_K \to D_3=\langle x,y \mid x^2=y^3=1, xyx=y^{-1}\rangle\]
by setting
\[ \begin{array}{rclrclrclrcl}
a&\mapsto& xy^2, 	&d&\mapsto& xy, 	&g&\mapsto& x, 	&j&\mapsto& xy, \\
b&\mapsto& x, 		&e&\mapsto& xy,  	&h&\mapsto& xy, 	&k&\mapsto& xy^2. \\
c&\mapsto& xy^2, 	&f&\mapsto& xy^2, 	&i&\mapsto& x,
\end{array} \]
The induced representation $\varrho\circ \al \co \pi_K \to \GL(3,\CC)$ is unitary and has a trivial summand given by the span of the vector $(1,1,1).$ Its orthogonal complement  gives a nonabelian rank 2 representation of $D_3$, which is necessarily equivalent to the unique irreducible $U(2)$ representation of $D_3.$ Using these observations, one can compute
the twisted Alexander polynomial using KnotTwister \cite{Fr12} and/or sage \cite{sage} to see that
\[ \De_K^{\varrho\circ \al}(t)=\frac{(3-11t+17t^{2}-11t^{3}+3t^{4})(3-13t^{2}+13t^{4}-3t^{6})}{(t-1)} \in \Bbb{Q}(t). \]
Note that
\[ 3-13t^{2}+13t^{4}-3t^{6}= - (t^2 - 1)(t^2 - 3)(3t^2 - 1),\]
which we can write as $f(t)f(-t)$ by
taking $$f(t) = (t+1)(t+\sqrt{3})(\sqrt{3}t+1).$$
However, it is not possible to  choose $f(t)\in \ZZ[t^{\pm 1}]$,
since $\ZZ[t^{\pm 1}]$ is a UFD and both $(t^2 - 3)$ and $(3t^2 - 1)$ are irreducible.
This gives the desired counterexample to Conjecture B (1).

\bigskip

Although Conjecture B (1) is in general false,  based on Proposition \ref{lem:conjboverc}, for a given knot $K$,
it is an interesting problem  to determine the minimal subring $R \subset \CC$
for which $g(t)$ can be factored as $g(t)=f(t)f(-t)$ for $f(t) \in R[t].$

%==========================================================
\subsection{Conjecture B (2)}

The following  proposition  provides a proof for
Conjecture B (2).

\begin{prop}
Let $p=2\ell+1$ be an odd prime and let $K$ be an oriented knot  together with an epimorphism
$\al \co \pi_K\to D_p$. Then
\[ \De_{K}^{\varrho \circ \al}(t)\equiv \left(\frac{\De_K(t)}{1-t}\right)^{\ell+1} \cdot
 \left(\frac{\De_K(-t)}{1+t}\right)^{\ell} \modd p.\]
\end{prop}

\begin{proof}
 Let $D_p$ be the dihedral group of order $2p$ from
Equation \eqref{eqndihedralgp} and set $V_p:=\FF_p[s^{\pm 1}]/(s^p-1)$.
Consider
the representation
$ \varrho_p \co D_p
\to \Aut_{\ZZ}(V_p) $
given by
\[ x\cdot q(s)=q(s^{-1})\text{ and } y\cdot q(s)=sq(s)\text{ for any }q(s)\in V.\]
It is straightforward to see that this representation is isomorphic to the mod $p$ reduction of the representation
$\varrho\co D_p\to  \GL(p,\ZZ)$ from the introduction.

It now follows from Lemma \ref{lem:propstap} that it suffices to prove the following claim:

\begin{claim}
\[ \De_{K}^{\varrho_p \circ \al}(t)\equiv \left(\frac{\De_K(t)}{1-t}\right)^{\ell+1} \cdot
 \left(\frac{\De_K(-t)}{1+t}\right)^{\ell} \modd p.\]
\end{claim}

For $i=0,\dots,p-1$ we define
\[ v_i:=\sum_{k=0}^{p-1}k^is^k \in V_p.\]
Here we write $0^0=0$. It follows from the Vandermonde determinant that
$v_0,v_1,\dots,v_{p-1}$ form a basis for $V_p$. We now consider the action of $D_p$ on $V_p$ which is given by $\varrho$.

\begin{claim}
For $i=0,\dots,p-1$ we have
\[  x\cdot v_i=\sum_{k\in \FF_p}(-k)^is^{k}=(-1)^iv_i \text{ and }y\cdot v_i=\sum_{j=0}^i \binom{i}{j} (-1)^{i-j}v_j. \]
\end{claim}

We first note that
\[ x\cdot v_i=x\cdot \sum_{k\in \FF_p}k^is^{k}=\sum_{k\in \FF_p}k^is^{p-k}=\sum_{k\in \FF_p}(-k)^is^{k}=(-1)^iv_i.\]
We now consider the action of $y$ on $v_i$. We have
\begin{eqnarray*} y\cdot v_i=s\cdot v_i &=& \sum_{k=0}^{p-1}k^is^{k+1} \\
&=&\sum_{k=0}^{p-1}(k-1)^is^{k}\\
&=&\sum_{k=0}^{p-1}\sum_{j=0}^i \binom{i}{j} k^j(-1)^{i-j}s^{k}\\
&=&\sum_{j=0}^i \binom{i}{j} (-1)^{i-j}\sum_{k=0}^{p-1}k^js^{k}\\
&=&\sum_{j=0}^i \binom{i}{j} (-1)^{i-j}v_j.
\end{eqnarray*}
This concludes the proof of the claim.

We now see that with respect to the basis $v_0,\dots,v_{p-1}$ the representation $D_p\to \Aut_{\FF_p}(V_p)$ is given by
\begin{equation} \label{equ:as}
 x\mapsto \begin{bmatrix} 1 &0&\dots &0 \\ 0&-1&\dots &0 \\ 0&0&\ddots&\vdots \\ 0&\dots&0&1\end{bmatrix}
 \text{ and }
y\mapsto \begin{bmatrix} 1 &*&\dots &* \\ 0&1&\dots &* \\ 0&0&\ddots&\vdots \\ 0&\dots&0&1\end{bmatrix}.\end{equation}
We now denote by $\ep\co \pi_K \to \GL(1,\FF_p)$ the trivial representation and  by
$\tau \co \pi_K \to \GL(1,\FF_p)$ the representation which is given by sending the meridian to $-1$.
Note that $\ga(x)=-1$ and $\ga(y)=1$.
It now follows from the definitions, from \eqref{equ:as} and from Lemma \ref{lem:propstap} that
\begin{eqnarray*}
\De_{K}^{\varrho_p \circ \al}(t) &\equiv& \prod_{i=1}^{\ell+1}\De_{K}^\ep(t)\cdot \prod_{i=1}^{\ell}\De_{K}^\tau(t) \;\; \modd p \\
&\equiv& \left(\frac{\De_K(t)}{1-t}\right)^{\ell+1} \cdot
 \left(\frac{\De_K(-t)}{1+t}\right)^{\ell} \;\; \modd p.
\end{eqnarray*}
\end{proof}

\end{document}